\documentclass{amsart}
\usepackage{tabularx}
\usepackage{graphicx}
\usepackage{amsmath}
\usepackage{amssymb}
\usepackage{amssymb}
\usepackage{amsbsy}
\usepackage{amsgen}
\usepackage{amsopn}
\usepackage{amscd}
\usepackage{amstext}
\usepackage{amsthm}

\theoremstyle{definition}
\newtheorem{definition}{Definition}

\theoremstyle{plain}
\newtheorem{theorem}{Theorem}
\newtheorem{lemma}{Lemma}
\newtheorem{corollary}{Corollary}

\theoremstyle{remark}
\newtheorem{remark}{Remark}

\theoremstyle{definition}
\newtheorem{example}{Example}
\newtheorem{acknow}{Acknowledgement}

\begin{document}
\title{(1, 2) and weak (1, 3) homotopies on knot projections}
\author{Noboru Ito \quad Yusuke Takimura}
%\date{\today}
\maketitle

\begin{abstract}
In this paper, we obtain the necessary and sufficient condition that two knot projections are related by a finite sequence of the first and second flat Reidemeister moves (Theorem \ref{12theorem}).  We also consider an equivalence relation that is called weak (1, 3) homotopy.  This equivalence relation occurs by the first flat Reidemeister move and one of the third flat Reidemeister moves.  We introduce a map sending weak (1, 3) homotopy classes to knot isotopy classes (Sec. \ref{weakInvariant}).  Using the map, we determine which knot projections are trivialized under weak (1, 3) homotopy (Corollary \ref{trivialCondition}).  
\end{abstract}

\section*{Addendum: added 2014}
After this article was published, the following information about doodles was pointed out by Roger Fenn. A doodle was introduced by Fenn and Taylor [2], which is a finite collection of closed curves without triple intersections on a closed oriented surface considered up to the second flat Reidemeister moves with the condition ($\ast$) that each component has no self-intersections. Khovanov [4] introduced doodle groups, and for his process, he considered doodles under a more generalized setting (i.e., removing the condition ($\ast$) and permitting the first flat Reidemeister moves). He showed [4, Theorem 2.2], a result similar to our [3, Theorem 2.2 (c)]. He also pointed out that [1, Corollary 2.8.9] gives a result similar to [4, Theorem 2.2].
The authors first noticed the above results by Fenn and Khovanov via personal communication with Fenn, and therefore, the authors would like to thank Roger Fenn for these references.

\section{Introduction}\label{intro}
Throughout this paper, we consider objects in the smooth category.  A {\it{knot}} is defined as a circle smoothly embedding into ${\mathbb{R}}^{3}$ and its {\it{knot projection}} is a {\it{regular projection}} of the knot to a sphere.  Here, the term {\it{regular projection}} means a projection to a sphere in which the image has only transversal double points of self-intersection.   If each double point of a knot projection is specified by over-crossing and under-crossing branches, we call the knot projection a {\it{knot diagram}}.  Therefore, for a given knot projection that has $n$ double points, it is possible to consider $2^{n}$ knot diagrams.  Indeed, knot projections have been studied by this approach \cite{taniyama1, taniyama2}.  

Knot isotopy classes are often interpreted as equivalence classes of knot diagrams under first, second, and third Reidemeister moves defined by Fig.~\ref{reidemeister}.  The diagrams in Fig.~\ref{reidemeister} show that the local replacements on the neighborhoods and the exterior of the neighborhoods are the same for both diagrams of each move.  
\begin{figure}[b]
\begin{picture}(0,0)
\put(23,50){$\Omega_1$}
\put(121,50){$\Omega_2$}
\put(245,50){$\Omega_3$}
\end{picture}
\includegraphics[width=11cm]{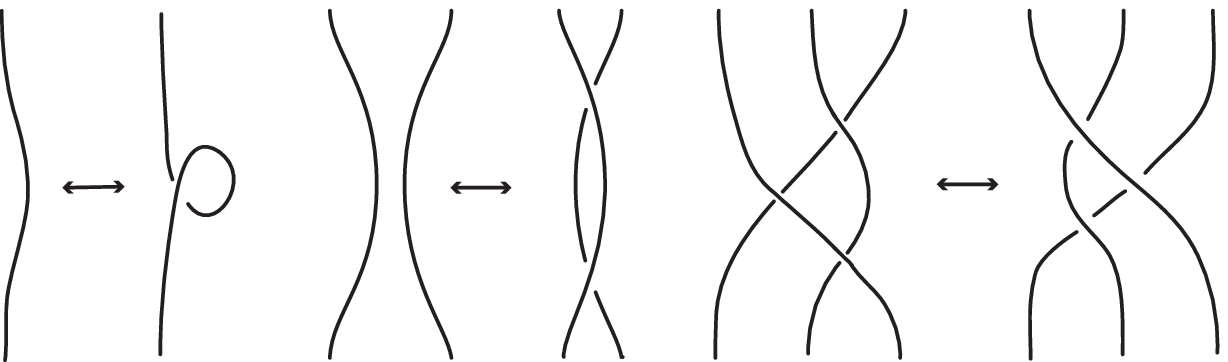}
\caption{Reidemeister moves $\Omega_1$, $\Omega_2$, and $\Omega_3$.  }\label{reidemeister}
\end{figure}
As shown in Fig.~\ref{ProjReidemeister}, we can define local moves of knot projections, called homotopy moves in this paper, by seeing projection images of Reidemeister moves of knot diagrams.  We call $H_1$, $H_2$, and $H_3$ of Fig.~\ref{ProjReidemeister} the first, second, and third homotopy moves or simply $H_1$, $H_2$, and $H_3$ moves.  
\begin{figure}[h]
\begin{picture}(0,0)
\put(23,52){$H_1$}
\put(130,52){$H_2$}
\put(253,52){$H_3$}
\end{picture}
\includegraphics[width=12cm,bb=0 0 369.33 100]{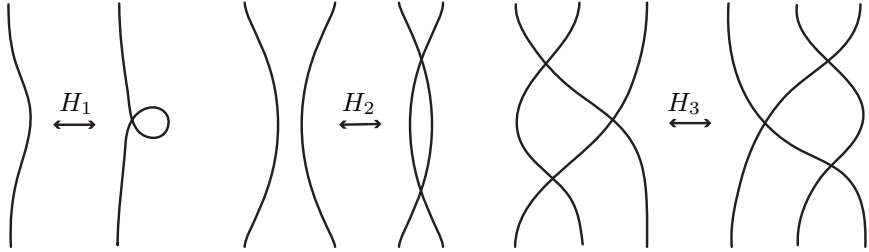}
\caption{Homotopy moves $H_1$, $H_2$, and $H_3$.  }\label{ProjReidemeister}
\end{figure}

Arnold introduced invariants of knot projections under the second or third homotopy moves, called {\it{perestroikas}} and found low-ordered invariants \cite{arnold1, arnold2} by using concepts similar to Vassiliev's ordered invariants of knots \cite{vassiliev}.  

This paper was motivated by attempts to solve the problem that determines which knot projections can be trivialized by the first and third moves.  Starting with Arnold's work, we can choose any two kinds of homotopy moves from the $H_1$, $H_2$, and $H_3$ moves.  First, we take $H_1$ and $H_2$ moves and consider the equivalence relation between knot projections by $H_1$ and $H_2$ moves.  We obtain a necessary and sufficient condition that two knot projections are equivalent under $H_1$ and $H_2$ moves (Theorem \ref{12theorem}).  Second, we take $H_2$ and $H_3$, which is nothing but regular homotopy, and knot projections under regular homotopy are classified by rotation numbers \cite{whitney}.  The last possibility is the case of $H_1$ and $H_3$ moves, which includes open problems: which two knot projections are equivalent under relations generated by $H_1$ and $H_3$ moves is unknown, and even which knot projection can be trivialized by $H_1$ and $H_3$ moves is unknown.  In this paper, as a first step, we give a necessary and sufficient condition that a knot projection can be trivialized under relations by $H_1$ and restricted $H_3$ moves.

The restricted $H_3$ moves are {\it{weak third homotopy moves}}, or simply weak $H_3$ moves, defined by Fig.~\ref{weakPerestroika} for knot projections.  
\begin{figure}[htbp]
\begin{picture}(0,0)
\put(44,45){\tiny weak $H_3$ move}
\put(210,45){\tiny strong $H_3$ move}
\end{picture}
\includegraphics[width=11cm,bb=0 0 435.67 107.5]{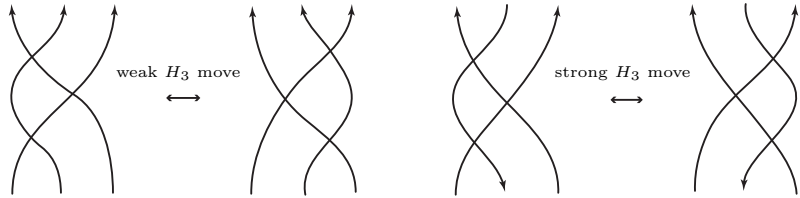}
\caption{Weak (left) and strong (right) third homotopy moves.  Viro \cite{viro} introduced these moves as weak (left) and strong (right) triple point perestroikas.  }\label{weakPerestroika}
\end{figure}

The weak third homotopy move is a positive weak perestroika, and its inverse move is based on the work of Viro \cite{viro} who defined high-ordered invariants by generalizing Arnold invariants.  Viro introduced {\it{weak}} and {\it{strong}} triple point perestroikas (Fig.~\ref{weakPerestroika}).  Arnold's (and subsequently Viro's) triple point perestroikas have positive directions that depend on the connecting branches of triple points (for details, see \cite[p. 6]{arnold2}).

For all knot projections, we consider equivalence classes under the first homotopy move and the weak third homotopy move, called {\it{weak (1, 3) homotopy}}.  In particular, we determine which knot projections are trivialized under weak (1, 3) homotopy (Corollary \ref{trivialCondition}).

At the end of this section, we discuss the remarkable work of Hagge and Yazinski \cite{HY} with regard to this paper.  We investigate the equivalence relation of knot projections by the first and third homotopy moves; i.e. {\it{(1, 3) homotopy}}.  Hagge and Yazinski \cite{HY} studied the non-triviality of (1, 3) homotopy classes of knot projections.  However, we still do not have any numerical invariants for all knot projections that exhibit the non-triviality of knot projections under (1, 3) homotopy.  

\section{(1, 2) homotopy classes of knot projections}
We define {\it{(1, 2) homotopy}} as the equivalence relation generated by $H_1$ and $H_2$ moves for all knot projections.  In this section, we determine how to detect two knot projections under (1, 2) homotopy (Theorem \ref{12theorem}).  The $H_1$ (resp. $H_2$) move consists of generating and removing a $1$-gon (resp. $2$-gon) as shown in Fig.~\ref{taki_reide1}.  Every generation (resp. removal) of a $1$-gon is called a $1a$ move or denoted by $1a$ (resp. $1b$ move or $1b$), and every generation (resp. removal) of a $2$-gon is called a $2a$ move or denoted by $2a$ (resp. $2b$ move or $2b$), as in Fig.~\ref{taki_reide1}.  
\begin{figure}[h]
\begin{picture}(0,0)
\put(30,35){$1b$}
\put(15,123){$1a$}
\put(125,123){$2a$}
\put(125,35){$2b$}
\end{picture}
\includegraphics[width=6cm]{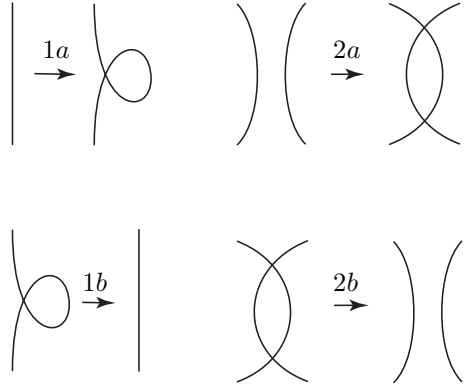}
\caption{Our conventions $1a$, $1b$, $2a$, and $2b$.  }\label{taki_reide1}
\end{figure}
\begin{definition}\label{1reduced}
For an arbitrary knot projection $P$ having $n$ individual $1$-gons and $2$-gons, the {\it{reduced projection}} $P^{r}$ is the knot projection obtained through any sequences of $1b$ and $2b$ deleting $n$ individual $1$-gons and $2$-gons arbitrarily.  

We consider two special cases of this definition to define {\it{1-homotopy}} (resp. {\it{2-homotopy}}) as the equivalence relation generated by $H_1$ (resp. $H_2$) moves for all knot projections.  For an arbitrary knot projection $P$ having $n$ individual $1$-gons (resp. $2$-gons), the {\it{reduced projection}} $P^{1r}$ (resp. $P^{2r}$) is the knot projection obtained through any sequences of $1b$ (resp. $2b$) deleting $n$ individual of $1$-gons (resp. $2$-gons).
\end{definition}
Definition \ref{1reduced} looks like that the definition depends on the way of deleting $1$- and $2$-gons (Fig.~\ref{takimuex}).  However, the reduced knot projection does not depend on the way of deleting $1$- and $2$-gons.  In other words, the uniqueness of the reduced knot projection is well defined by Theorem \ref{12theorem} (Corollary \ref{12uniqueness}).  
\begin{figure}[h]
\begin{picture}(0,0)
\put(28,-10){$P_2$}
\put(28,70){$P_1$}
\put(113,70){${P_1}^{r}$}
\put(113,-10){${P_2}^{r}$}
\end{picture}
\includegraphics[width=5cm]{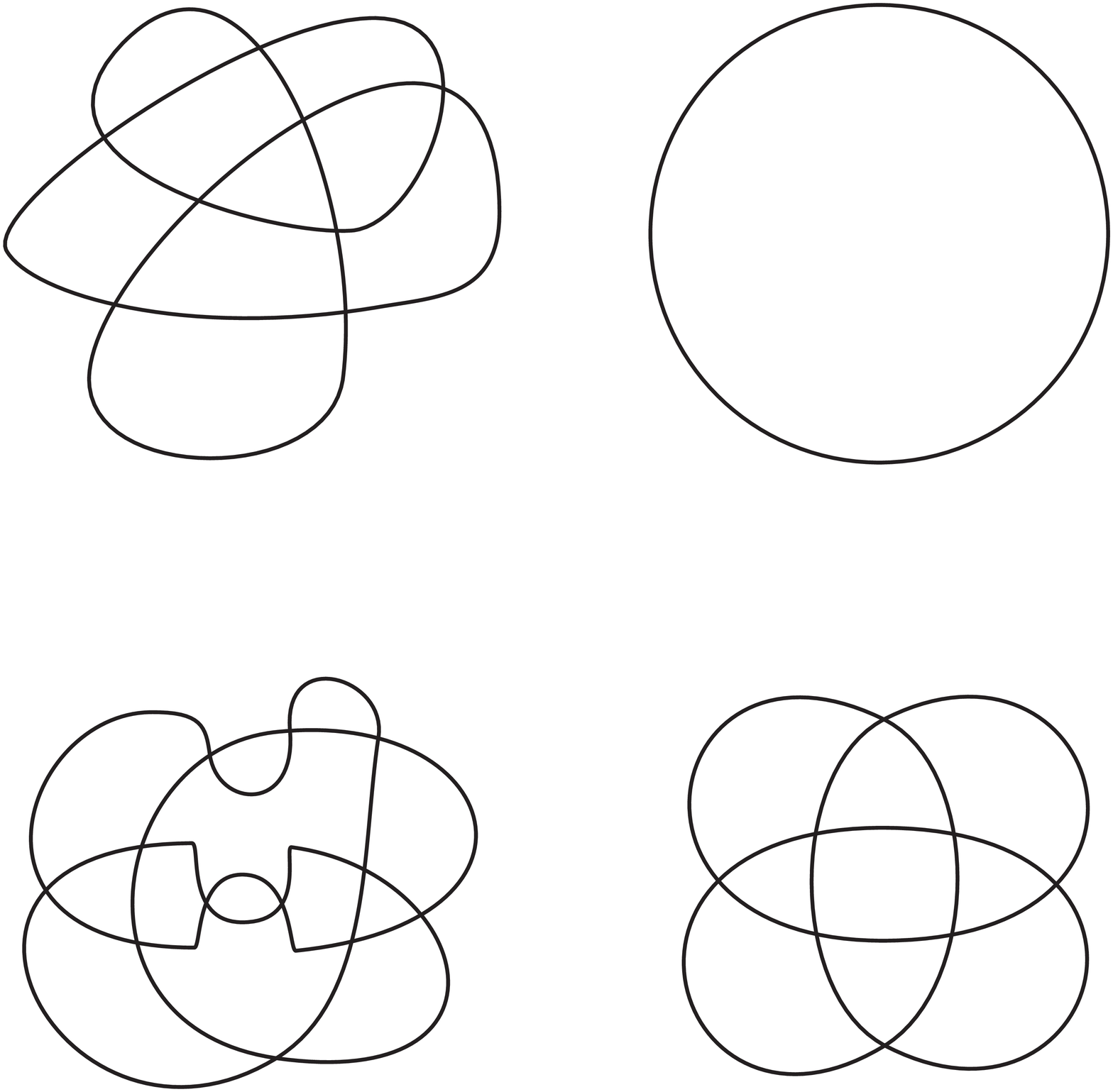}
\caption{Examples of knot projections and their reduced knot projections (by using Theorem \ref{12theorem}, ${P_1}^{r} \neq {P_2}^{r}$ under (1, 2) homotopy).}\label{takimuex}
\end{figure}
\begin{theorem}\label{12theorem}
\noindent
\begin{enumerate}
\item Two knot projections $P$ and $P'$ are equivalent under 1-homotopy if and only if $P^{1r}$ and ${P'}^{1r}$ are equivalent under isotopy on $S^{2}$.  \label{1stStatement}
\item Two knot projections $P$ and $P'$ are equivalent under 2-homotopy if and only if $P^{2r}$ and ${P'}^{2r}$ are equivalent under isotopy on $S^{2}$.  \label{2ndStatement}
\item Two arbitrary knot projections $P$ and $P'$ are equivalent under (1, 2) homotopy if and only if ${P}^{r}$ and ${P'}^{r}$ are equivalent under isotopy on $S^{2}$.  \label{3rdStatement}
\end{enumerate}
\end{theorem}
We will obtain the proof of (\ref{3rdStatement}) which includes the proofs of (\ref{1stStatement}) and (\ref{2ndStatement}).  As shown below, Theorem \ref{12theorem} is derived from Lemma \ref{12lemma}.  
\begin{lemma}\label{12lemma}
Any finite sequence generated by $H_1$ and $H_2$ moves between an arbitrary knot projection and an arbitrary reduced knot projection can be replaced with a sequence of only $1a$ and $2a$ moves or only $1b$ and $2b$ moves.  
\end{lemma}
\begin{proof}
Let $n$ be an arbitrary integer, with $n \ge 2$, and let $x$ be a sequence of $n-2$ moves consisting of $1a$ and $2a$.  We use the convention that the sequence $x$ followed by a $1a$ move is denoted by $x(1a)$.  For the other moves (e.g. $1b$, $2a$, or $(2a)(1b)$), the same convention is applied.  Let $P_i$ be the $i$th knot projection appearing in a sequence of $H_1$ and $H_2$ moves of length $n$.  In the discussion below, we often use the symbol $Q$, which stands for a knot projection.  We also use the convention that if the sequence $x(1a)(1b)$ can be replaced with $x$, we denote this by $x(1a)(1b)$ $=$ $x$.  We apply the same convention to all similar cases that appear below.  

Below we make claims about four cases of the first appearance of $1b$ or $2b$ in the sequence $P_1$ $\to$ $P_2$ $\to \dots \to$ $P_n$ $\to$ $P_{n+1}$ of $H_1$ and $H_2$ moves.  
\begin{itemize}
\item Case 1: $x(1a)(1b)$ $=$ $x$ or $x(1b)(1a)$.  
\item Case 2: $x(2a)(1b)$ $=$ $x(1a)$ or $x(1b)(2a)$.  
\item Case 3: $x(1a)(2b)$ $=$ $x(1b)$ or $x(2b)(1a)$.  
\item Case 4: $x(2a)(2b)$ $=$ $x$ or $x(2b)(2a)$.  
\end{itemize}

Case 1: The last two moves $(1a)(1b)$ can be presented as in Fig.~\ref{taki1-1}.  The symbols $\delta_x$ and $\delta_y$ denote $1$-gons with boundaries, as shown in Fig.~\ref{taki1-1}.  
\begin{figure}[h]
\begin{picture}(0,0)
\put(45,16){\tiny $\delta_x$}
\put(45,73){\tiny $\delta_y$}
\put(20,25){$1a$}
\put(65,80){$1b$}
\end{picture}
\includegraphics[width=3cm]{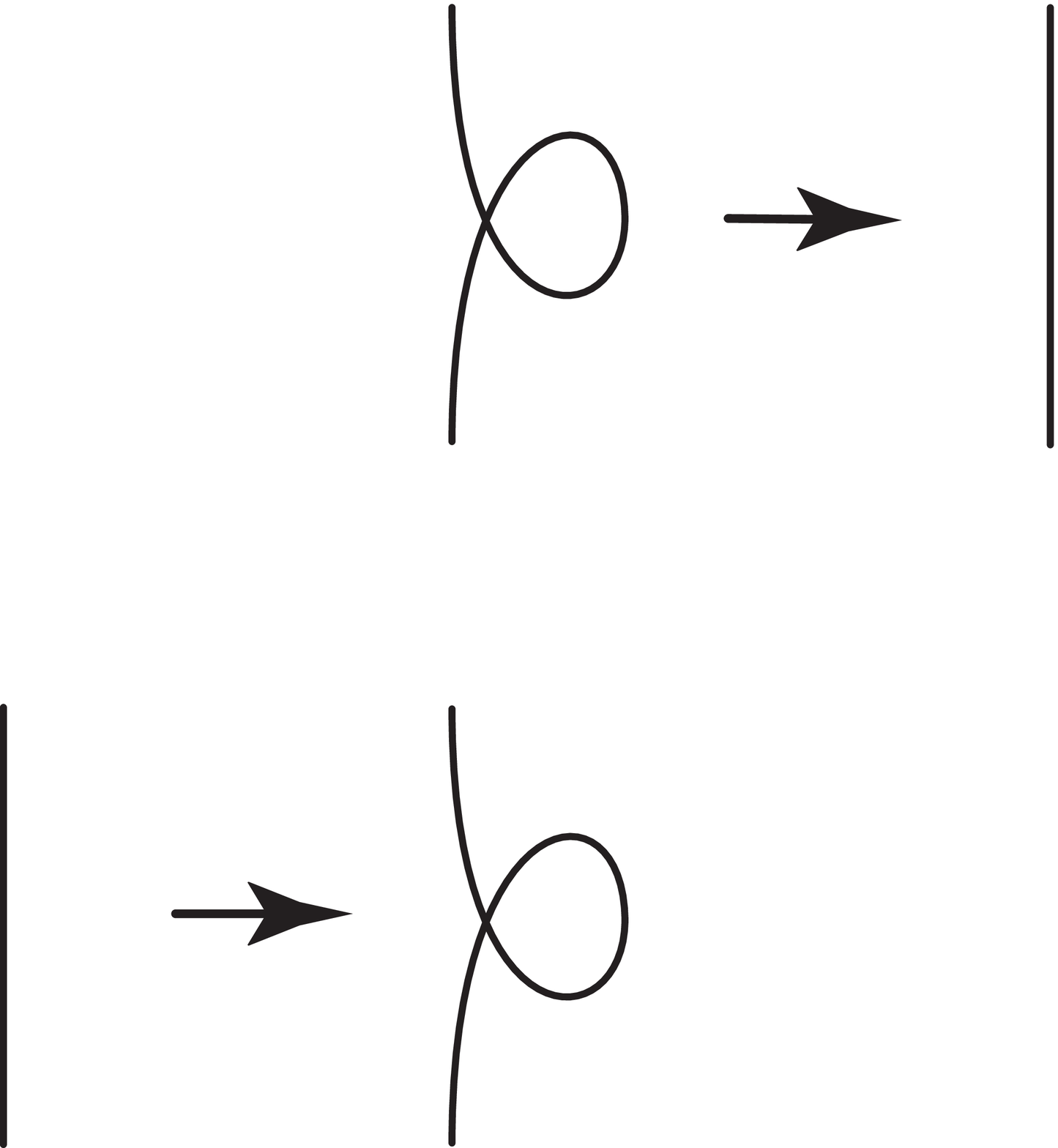}
\caption{The last two moves of $x(1a)(1b)$ of Case 1.}\label{taki1-1}
\end{figure}
\begin{enumerate}
\item If $\delta_x$ $\cap$ $\delta_y$ $\neq$ $\emptyset$, there are two cases of the pair $\delta_x$ and $\delta_y$ as shown in Fig.~\ref{taki1-2}.    
\begin{figure}[h]
\includegraphics[width=3cm]{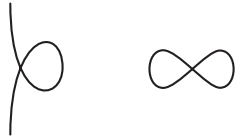}
\caption{Case 1--(i).  The case $\delta_x$ $=$ $\delta_y$ (left) and the case $\delta_x$ $\cap$ $\delta_y$ $=$ $\{{\text{one vertex}}\}$ (right).}\label{taki1-2}
\end{figure}
In both of these cases, we have $x(1a)(1b)$ $=$ $x$ by Fig.~\ref{taki1-3}.
\begin{figure}[h]
\includegraphics[width=3cm]{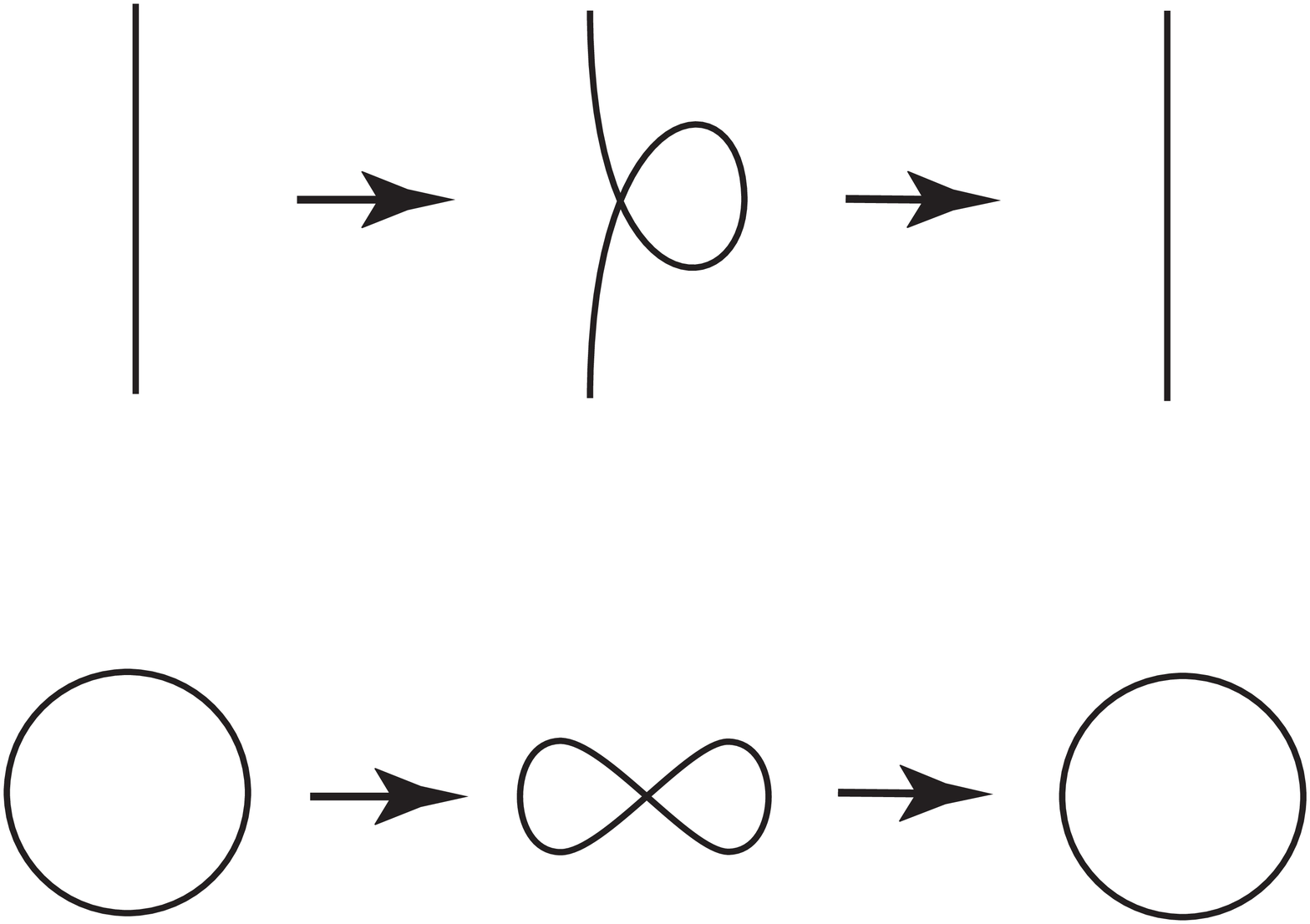}
\caption{Case 1--(i).  The sequence $x(1a)(1b)$ $=$ $x$.}\label{taki1-3}
\end{figure}
\item If $\delta_x$ $\cap$ $\delta_y$ $=$ $\emptyset$, then Fig.~\ref{taki1-4} implies $x(1a)(1b)(1b)$ $=$ $x(1b)$.  Here, we can find the special move $1b$, corresponding to the inverse move of given $1a$.  The move $1a$ follows this sequence, and we have $x(1a)(1b)(1b)(1a)$ $=$ $x(1b)(1a)$ as in Fig.~\ref{taki1-4}.  Therefore, $x(1a)(1b)$ $=$ $x(1b)(1a)$.  
\begin{figure}[h]
\begin{picture}(0,0)
\put(0,79){$P_{n-1}$}
\put(17,87){\tiny ${1a}$}
\put(30,79){$P_n$}
\put(50,87){\tiny ${1b}$}
\put(60,79){$P_{n+1}$}
\put(87,87){\tiny $1b$}
\put(105,79){$Q$}
\put(120,87){\tiny $1a$}
\put(130,79){$P_{n+1}$}
\put(50,110){\tiny $1b$}
\end{picture}
\includegraphics[width=5cm]{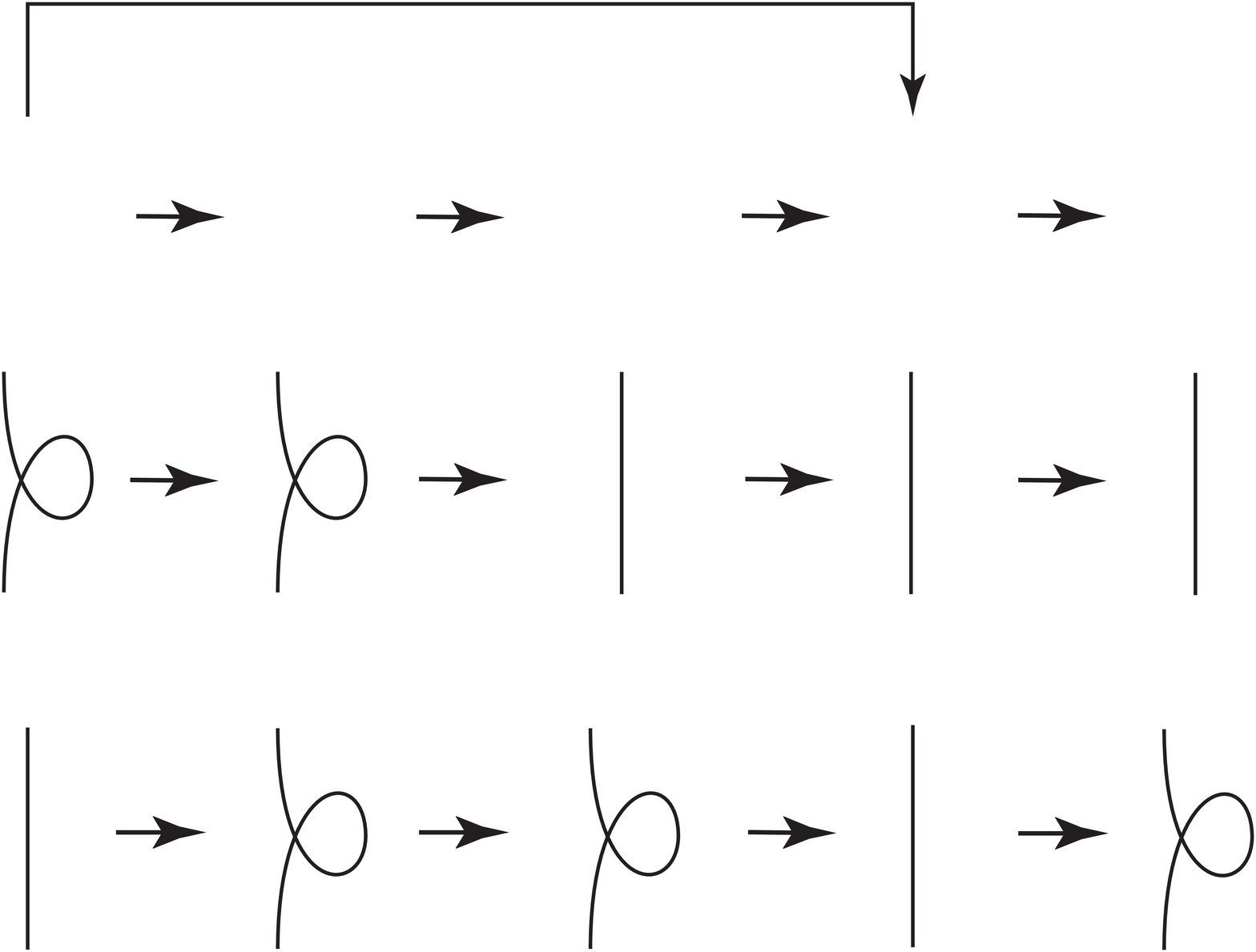}
\caption{Case 1--(ii).  The sequence $P_{n-1}$ $\stackrel{1a}{\to}$ $P_{n}$ $\stackrel{1b}{\to}$ $P_{n+1}$ $\stackrel{1b}{\to}$ $Q$ shows that $x(1a)(1b)(1b)$ $=$ $x(1b)$.}\label{taki1-4}
\end{figure}
\end{enumerate}

Case 2: The last two moves $(2a)(1b)$ can be presented as in Fig.~\ref{taki2-1}.  The symbol $\delta_x$ (resp. $\delta_y$) denotes a $2$-gon (resp. $1$-gon) with a boundary, as shown in Fig.~\ref{taki2-1}.  
\begin{figure}[h]
\begin{picture}(0,0)
\put(46,16){\tiny $\delta_x$}
\put(49,57){\tiny $\delta_y$}
\put(22,22){$2a$}
\put(65,65){$1b$}
\end{picture}
\includegraphics[width=3cm]{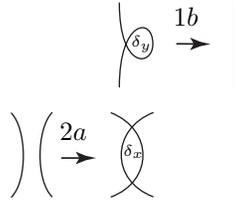}
\caption{The last two moves of $x(2a)(1b)$ of Case 2.}\label{taki2-1}
\end{figure}
\begin{enumerate}
\item If $\delta_x$ $\cap$ $\delta_y$ $\neq$ $\emptyset$, the pair $\delta_x$ and $\delta_y$ appears as in Fig.~\ref{taki2-2}.  
\begin{figure}[h]
\begin{picture}(0,0)
\put(9,18){\tiny $\delta_x$}
\put(23,19){\tiny $\delta_y$}
\end{picture}
\includegraphics[width=1cm]{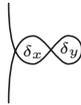}
\caption{Case 2-(i).  }\label{taki2-2}
\end{figure}
In this case, we have $x(2a)(1b)$ $=$ $x(1a)$ by Fig.~\ref{taki2-3}.  
\begin{figure}[h]
\includegraphics[width=3cm]{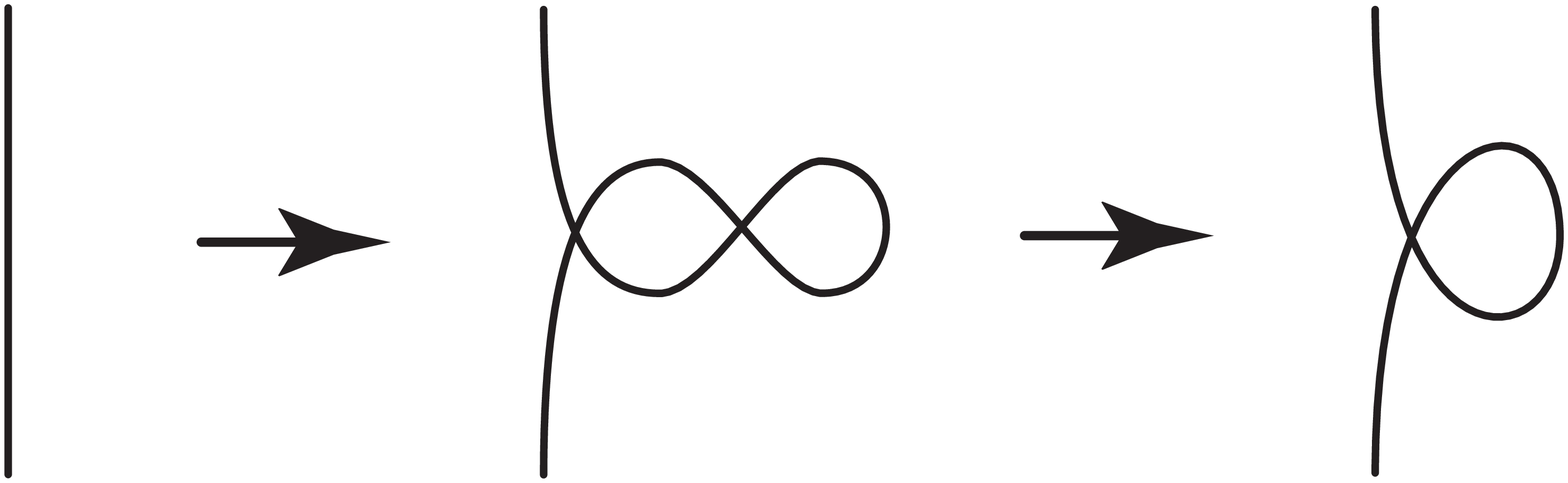}
\caption{Case 2-(i).  The sequence $x(2a)(1b)$ $=$ $x(1a)$.}\label{taki2-3}
\end{figure}
\item If $\delta_x$ $\cap$ $\delta_y$ $=$ $\emptyset$, then Fig.~\ref{taki2-4} implies that $x(2a)(1b)(2b)$ $=$ $x(1b)$.  Here, we can find the special move $2b$, corresponding to the inverse move of given $2a$.  The move $2a$ follows this sequence, and we have $x(2a)(1b)(2b)(2a)$ $=$ $x(1b)(2a)$ as in Fig.~\ref{taki2-4}.  Therefore, $x(2a)(1b)$ $=$ $x(1b)(2a)$.  
\begin{figure}[h]
\begin{picture}(0,0)
\put(0,59){$P_{n-1}$}
\put(18,67){\tiny $2a$}
\put(30,59){$P_n$}
\put(47,67){\tiny $1b$}
\put(60,59){$P_{n+1}$}
\put(80,67){\tiny $2b$}
\put(93,59){$Q$}
\put(107,67){\tiny $2a$}
\put(120,59){$P_{n+1}$}
\put(50,90){\tiny $1b$}
\end{picture}
\includegraphics[width=4.5cm]{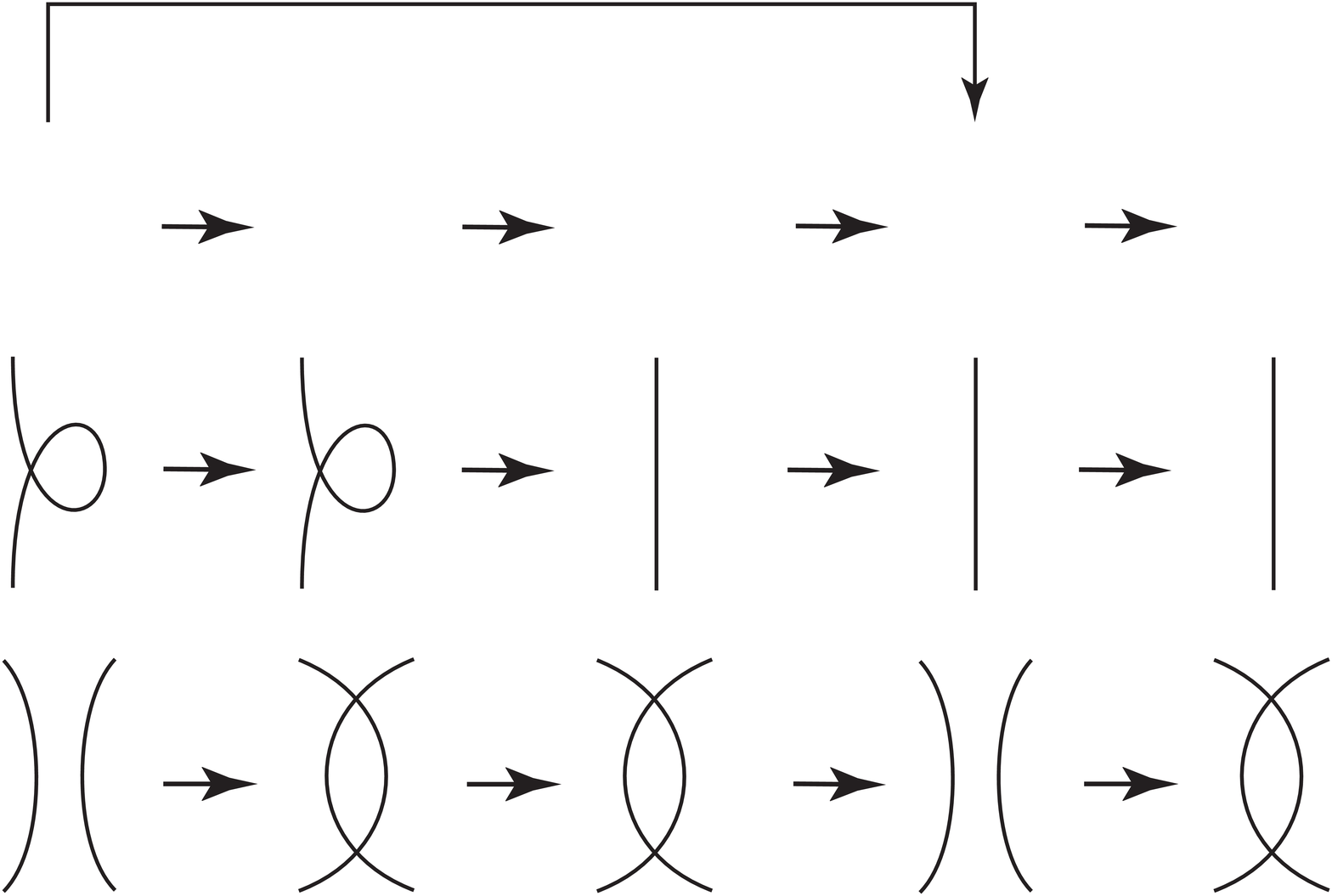}
\caption{Case 2-(ii).  The sequence $P_{n-1}$ $\stackrel{2a}{\to}$ $P_{n}$ $\stackrel{1b}{\to}$ $P_{n+1}$ $\stackrel{2b}{\to}$ $Q$ shows that $x(2a)(1b)(2b)$ $=$ $x(1b)$.}\label{taki2-4}
\end{figure}
\end{enumerate}

Case 3: The last two moves $(1a)(2b)$ can be presented as in Fig.~\ref{taki3-1}.  The $\delta_x$ (resp. $\delta_y$) denotes a $1$-gon ($2$-gon) with a boundary, as shown in Fig.~\ref{taki3-1}.  
\begin{figure}[h]
\begin{picture}(0,0)
\put(41,15){\tiny $\delta_x$}
\put(37,57){\tiny $\delta_y$}
\put(16,23){$1a$}
\put(57,63){$2b$}
\end{picture}
\includegraphics[width=3cm]{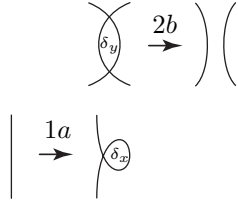}
\caption{The last two moves of $x(1a)(2b)$ of Case 3.}\label{taki3-1}
\end{figure}
\begin{enumerate}
\item If $\delta_x$ $\cap$ $\delta_y$ $\neq$ $\emptyset$, the pair $\delta_x$ and $\delta_y$ appears as in Fig.~\ref{taki3-2}.    
\begin{figure}[h]
\begin{picture}(0,0)
\put(9,19){\tiny $\delta_y$}
\put(23,18){\tiny $\delta_x$}
\end{picture}
\includegraphics[width=1cm]{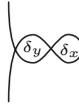}
\caption{Case 3-(i).}\label{taki3-2}
\end{figure}
In this case, $x(1a)(2b)$ $=$ $x(1b)$ by Fig.~\ref{taki3-3}.
\begin{figure}[h]
\includegraphics[width=3cm]{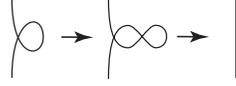}
\caption{Case 3-(i).  The sequence $x(1a)(2b)$ $=$ $x(1b)$.  }\label{taki3-3}
\end{figure}
\item If $\delta_x$ $\cap$ $\delta_y$ $=$ $\emptyset$, then Fig.~\ref{taki3-4} implies that $x(1a)(2b)(1b)$ $=$ $x(2b)$.  Here, we can find the special move $1b$, corresponding to the inverse move of given $1a$.  The move $1a$ follows this sequence, and we have $x(1a)(2b)(1b)(1a)$ $=$ $x(2b)(1a)$ as in Fig.~\ref{taki3-4}.  Therefore, $x(1a)(2b)$ $=$ $x(2b)(1a)$.  
\begin{figure}[h]
\begin{picture}(0,0)
\put(-4,63){$P_{n-1}$}
\put(17,69){\tiny $1a$}
\put(30,63){$P_n$}
\put(48,69){\tiny $2b$}
\put(58,63){$P_{n+1}$}
\put(79,69){\tiny $1b$}
\put(92,63){$Q$}
\put(107,69){\tiny $1a$}
\put(122,63){$P_{n+1}$}
\put(50,88){\tiny $2b$}
\end{picture}
\includegraphics[width=4.5cm]{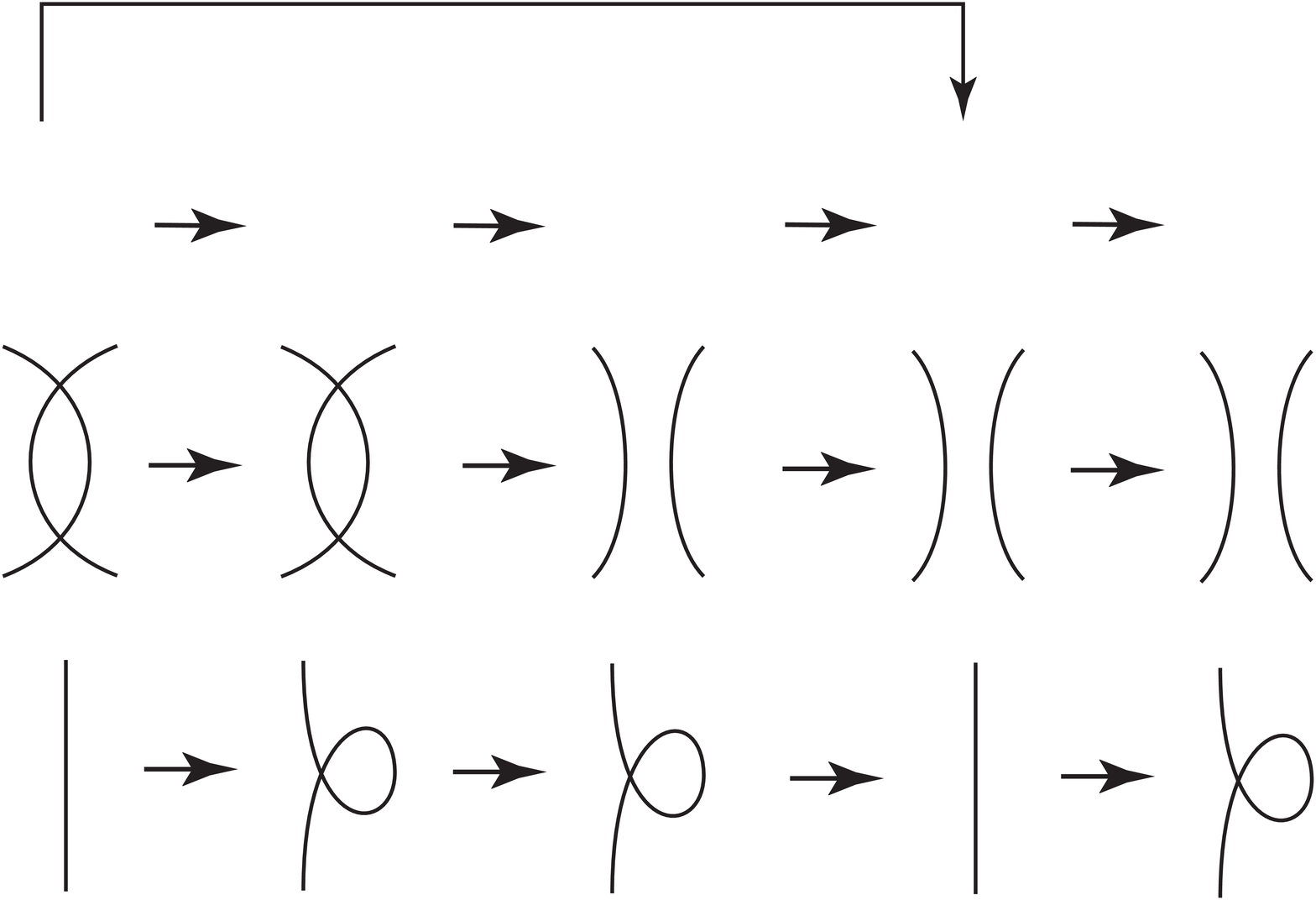}
\caption{Case 3-(ii).  The sequence $P_{n-1}$ $\stackrel{1a}{\to}$ $P_{n}$ $\stackrel{2b}{\to}$ $P_{n+1}$ $\stackrel{1b}{\to}$ $Q$ shows that $x(1a)(2b)(1b)$ $=$ $x(2b)$.}\label{taki3-4}
\end{figure}
\end{enumerate}

Case 4: The last two moves $(2a)(2b)$ can be presented as in Fig.~\ref{taki4-1}.  The symbols $\delta_x$ and $\delta_y$ denote $2$-gons with boundaries, as shown in Fig.~\ref{taki4-1}.  
\begin{figure}[h]
\begin{picture}(0,0)
\put(43,16){\tiny $\delta_x$}
\put(43,60){\tiny $\delta_y$}
\put(23,20){$2a$}
\put(58,65){$2b$}
\end{picture}
\includegraphics[width=3cm]{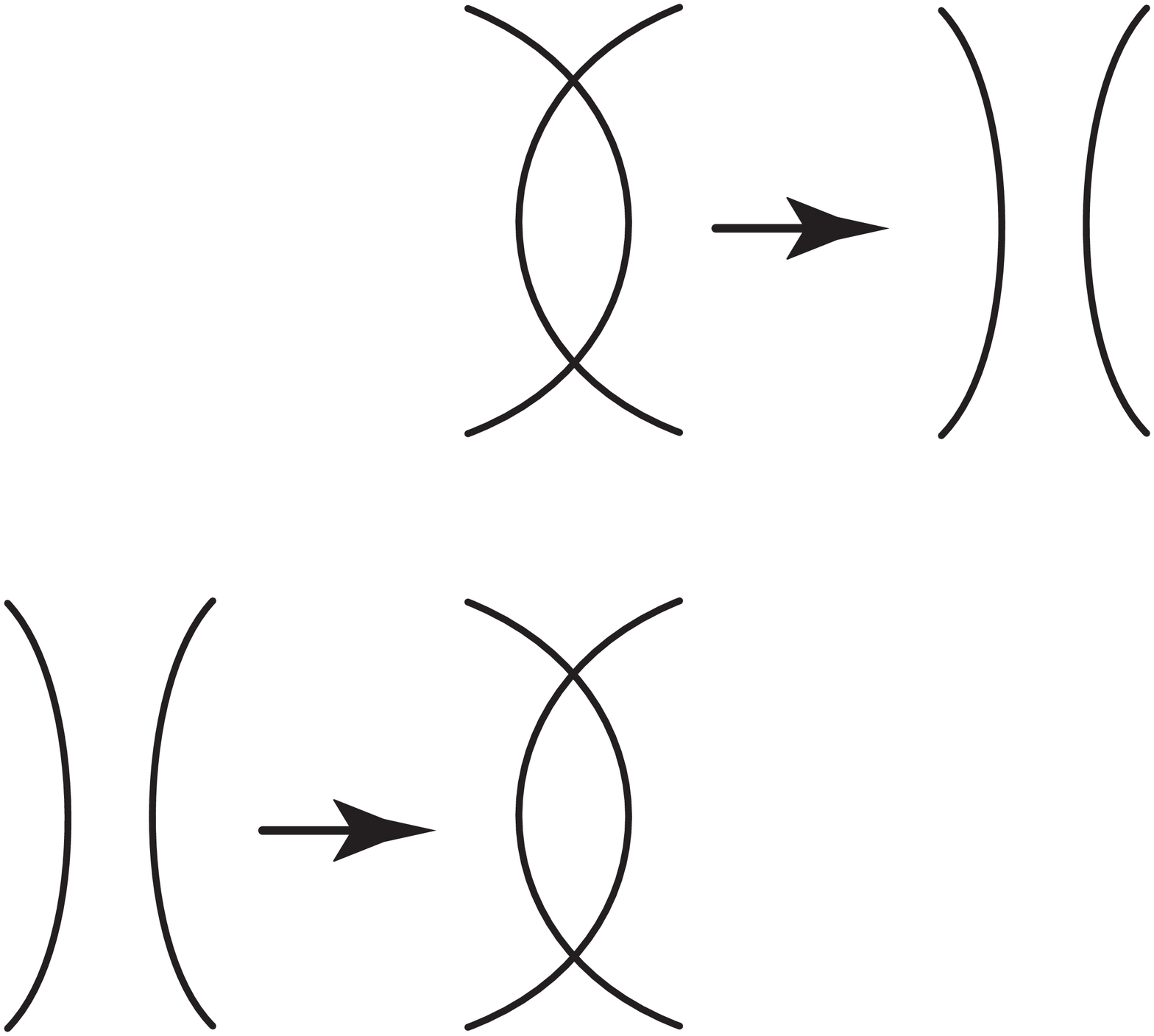}
\caption{The last two moves of $x(2a)(2b)$ of Case 4.}\label{taki4-1}
\end{figure}
\begin{enumerate}
\item If $\delta_x$ $\cap$ $\delta_y$ $\neq$ $\emptyset$, there are two cases of the pair $\delta_x$ and $\delta_y$, as shown in Fig.~\ref{taki4-2}.    
\begin{figure}[h]
\begin{picture}(0,0)
\put(58,16){\tiny $\delta_x$}
\put(73,16){\tiny $\delta_y$}
\end{picture}
\includegraphics[width=3cm]{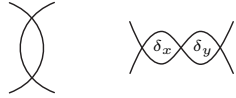}
\caption{Case 4-(i).  The case $\delta_x$ $=$ $\delta_y$ (left) and the case $\delta_x$ $\cap$ $\delta_y$ $=$ $\{{\text{one vertex}}\}$.}\label{taki4-2}
\end{figure}
In both of these cases, we have $x(2a)(2b)$ $=$ $x$ as in Fig.~\ref{taki4-3}. 
\begin{figure}[h]
\begin{picture}(0,0)
\put(64,8){$\delta_x$}
\put(81,8){$\delta_y$}
\end{picture}
\includegraphics[width=5cm]{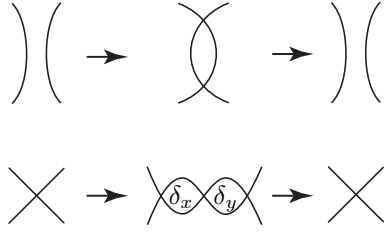}
\caption{Case 4-(i).  The sequence $x(2a)(2b)$ $=$ $x$.  }\label{taki4-3}
\end{figure}
\item If $\delta_x$ $\cap$ $\delta_y$ $=$ $\emptyset$, then Fig.~\ref{taki4-4} implies that $x(2a)(2b)(2b)$ $=$ $x(2b)$.  Here, we can find the special move $2b$, corresponding to the inverse move of given $2a$.  The move $2a$ follows this sequence, and we have $x(2a)(2b)(2b)(2a)$ $=$ $x(2b)(2a)$ as in Fig.~\ref{taki4-4}.  Therefore, $x(2a)(2b)$ $=$ $x(2b)(2a)$.  
\end{enumerate}
\begin{figure}[h]
\begin{picture}(0,0)
\put(-4,61){$P_{n-1}$}
\put(16,66){\tiny $2a$}
\put(30,61){$P_n$}
\put(48,66){\tiny $2b$}
\put(58,61){$P_{n+1}$}
\put(79,66){\tiny $2b$}
\put(93,61){$Q$}
\put(107,66){\tiny $2a$}
\put(118,61){$P_{n+1}$}
\put(50,87){\tiny $2b$}
\end{picture}
\includegraphics[width=4.5cm]{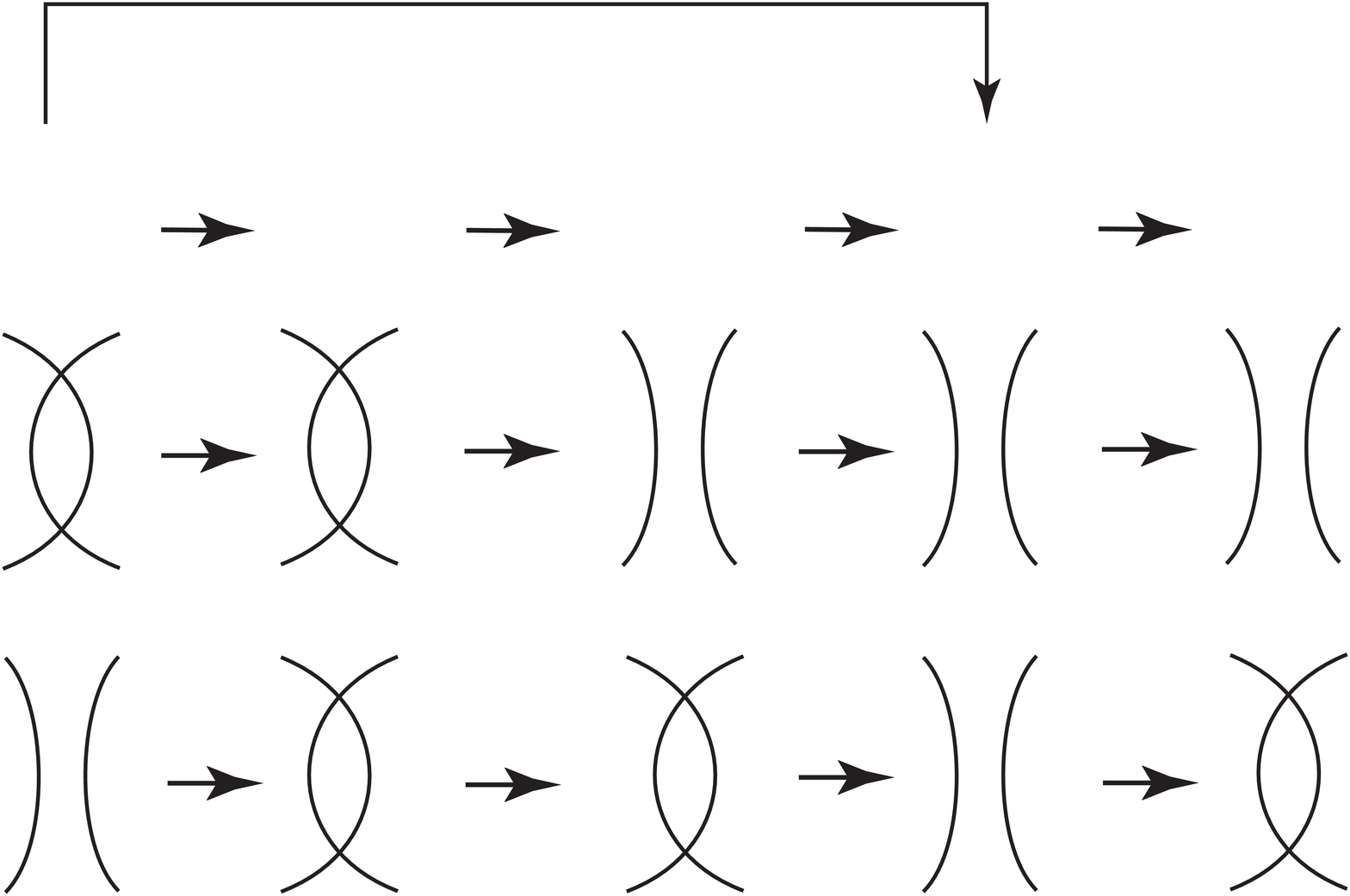}
\caption{Case 4-(ii).  The sequence $P_{n-1}$ $\stackrel{2a}{\to}$ $P_n$ $\stackrel{2b}{\to}$ $P_{n+1}$ $\stackrel{2b}{\to}$ $Q$ shows that $x(2a)(2b)(2b)$ $=$ $x(2b)$.}\label{taki4-4}
\end{figure}

Thus, we have shown that the above claims about the four cases are true.  

Next, we show that the statement of Lemma \ref{12lemma} is true.  Let us consider a given sequence $s$ of $1a$, $1b$, $2a$, and $2b$ from the left with the given reduced knot projection.  Here, $1b$ and $2b$ are called $b$ moves.  We focus on the first appearance of any $b$ move, which is called the first $b$ move.  The first $b$ move cannot be the first move of the sequence $s$ since we start from the left with the reduced knot projection that does not have any $1$-gons and $2$-gons.  If the first $b$ move is $1b$, we use the discussions of Cases 1-(ii) and 2-(ii) and move to the left (if necessary) until the $b$ move encounters Case 1-(i) or 2-(i), either of which eliminates the $b$ moves.  Here, note that the $b$ move must encounter Case 1-(i) or 2-(i) because the $b$ move must not the first move of $s$.  If the first $b$ move is $2b$, we use the discussions of Cases 3-(ii) and 4-(ii) and move to the left (if necessary) until the $b$ move encounters Case 3-(i) or 4-(i), either of which eliminates the $2b$ move.  Cases 3-(i) or 4-(i) eliminate $2b$ and permit the replacement $y(2b)$ with $z(1b)$, where $y$ and $z$ are sequences entirely consisting of $1a$ and $2a$ moves.  However, $z(1b)$ belongs to Case 1 or 2, and so the $b$ move $1b$ is eliminated.  This completes the proof.  
\end{proof}

Now we will prove Theorem \ref{12theorem}.  
\begin{proof}
For two arbitrary knot projections $P$ and $P'$, we take the projections $P^{r}$ and ${P'}^{r}$ arbitrarily.  For $P^{r}$ and ${P'}^{r}$, we apply Lemma \ref{12lemma}.  If $P^{r}$ $\neq$ ${P'}^{r}$ under isotopy on $S^{2}$, there exists a non-empty sequence of only $1a$ and $2a$ moves from $P^{r}$ to ${P'}^{r}$ or from ${P'}^{r}$ to $P^{r}$.  However, neither $P^{r}$ nor ${P'}^{r}$ has any $1$-gon or $2$-gon.  This contradicts that the non-empty sequence consists of only $1a$ and $2a$ moves.  Then, the assumption that $P^{r}$ $\neq$ ${P'}^{r}$ under isotopy on $S^{2}$ is false.  Therefore, our claim is true.  The proof of the statement (\ref{1stStatement}) (resp. (\ref{2ndStatement})) of Theorem \ref{12theorem} is obtained by considering Case 1 (resp. Case 4) of Lemma \ref{12lemma}.  
\end{proof}
Theorem \ref{12theorem} implies Corollary \ref{12uniqueness}.  
\begin{corollary}\label{12uniqueness}
For an arbitrary knot projection $P$, the reduced knot projection $P^{r}$ is uniquely determined; that is, $P^{r}$ does not depend on the way in which $1$-gons and $2$-gons of $P$ are deleted.  
\end{corollary}
\begin{remark}
Reduced projections having no $1$-gons and $2$-gons, produced by Definition \ref{1reduced}, are called {\it{lune-free graphs}} \cite{EHK, AST}.  The knot projection ${P_{2}}^{r}$ of Fig.~\ref{takimuex} appears in \cite{EHK, AST}.  
\end{remark}

\section{Positive resolutions and weak (1, 3) homotopy invariants.}\label{weakInvariant}

In the rest of this paper, unless otherwise specified, we adopt unoriented knot projections, and so the sphere containing knot projections does not have its orientation.  Moreover, by invoking isotopy on $S^{2}$, if necessary, we can assume without loss of generality that every double point of the knot projection $P$ consists of two orthogonal branches.  

In this section, we define a map from weak (1, 3) homotopy classes to knot isotopy classes.  Take an arbitrary knot projection $P$ and give $P$ any orientation.  Let us define {\it{crossings}} as double points of knot diagrams.  We replace the neighborhood of every double point by that of the crossing of knot diagrams as shown in Fig.~\ref{positiveResolution}.  
\begin{figure}[h!]
\includegraphics[width=8cm]{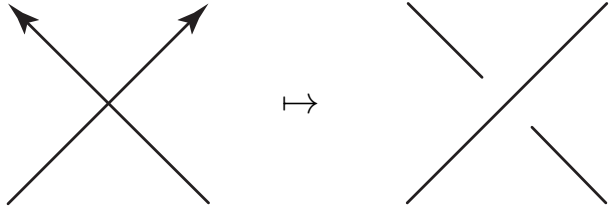}
\begin{picture}(0,0)
\put(-127,35){\LARGE$\mapsto$}
\end{picture}
\caption{Positive resolution.}\label{positiveResolution}
\end{figure}
This replacement does not depend on the orientation of $P$.  Then, the replacements define the map from knot projections to knot diagrams.  

Polyak \cite {polyak} introduced this map and called each replacement a {\it{positive resolution}} when every double point is regarded as a singular point.  That is, this map gives resolutions of singularities of double points.  Using this map, Polyak defined finite type invariants of plane curves.  We consider further applications.  

\begin{theorem}\label{theorem1}
For an arbitrary knot projection $P$, the positive resolution of all double points of $P$ defines the map from weak (1, 3) homotopy classes of knot projections to knot diagrams.   
\end{theorem}
\begin{proof}
Let us denote by $p$ the map defined by the positive resolutions of all double points of $P$ and sending knot projections to knot diagrams.  We will check the behavior of $p$ of the first and third homotopy moves.  
\begin{itemize}
\item $H_1$ moves.

We denote by $D_1$ (resp. $D_2$) the local diagram defined by the left (resp. right) side of the $H_1$ move in Fig.~\ref{ProjReidemeister}.  All possibilities for $D_2$ are shown in Fig.~\ref{firstMove}.  In every case, $p(D_2)$ is transferred to $p(D_1)$ by the first Reidemeister move $\Omega_1$ of Fig.~\ref{reidemeister}.  

\begin{figure}[h]
\begin{picture}(0,0)
\put(43,23){$\mapsto$}
\put(137,23){$\mapsto$}
\end{picture}
\includegraphics[width=6cm]{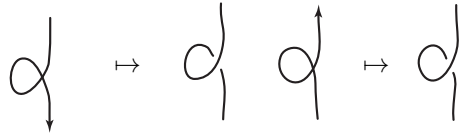}
\caption{The first homotopy move $H_1$ and positive resolutions.}\label{firstMove}
\end{figure}
\item Weak $H_3$ moves.  

We denote by $D_3$ (resp. $D_4$) the local diagram defined by the left (resp. right) side of the weak $H_3$ move in Fig.~\ref{weakPerestroika}.  As Fig.~\ref{thirdMove} shows, $p(D_3)$ is transferred to $p(D_4)$ by the third Reidemeister move $\Omega_3$ of Fig.~\ref{reidemeister}.  
\begin{figure}
\begin{picture}(0,0)
\put(60,35){$\mapsto$}
\put(220,35){$\mapsto$}
\end{picture}
\includegraphics[width=10cm, bb=0 0 400 100]{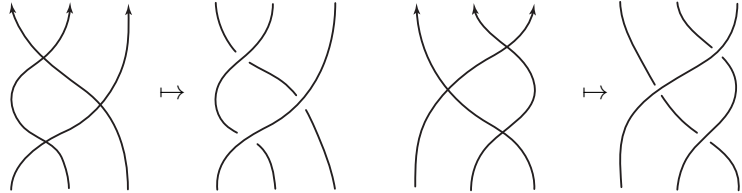}
\caption{The third homotopy move $H_3$ and positive resolutions.}\label{thirdMove}
\end{figure}
\end{itemize}
Then, an arbitrary representative of a weak (1, 3) homotopy class is sent by the map $p$ to an isotopy class of a knot.  This completes the proof.  
\end{proof}

In what follows, we permit the same symbol $p$ to denote the map defined by Theorem \ref{theorem1} from weak (1,3) homotopy classes of knot projections to isotopy classes of knots.  
\begin{corollary}
Let $I$ be an arbitrary knot invariant and let $p$ be the map defined by Theorem \ref{theorem1} from weak (1,3) homotopy classes of knot projections to knot isotopy classes.  Then, $I \circ p$ is an invariant under weak (1, 3) homotopy.  
\end{corollary}

\begin{example}
As is well known, tricolorability is a knot invariant.  Let $I_1$ be the map $I_1(D)$ $=$ $1$ if a knot diagram $D$ is tricolorable and the map $I_1(D)$ $=$ $0$ otherwise.  
Let $P_5$, $P_6$, and $P_7$ be the knot projections shown from the left to right as in Fig.~\ref{knotProjections}.  Thus, $I_1 \circ p (P_5)$ $=$ $0$, $I_1 \circ p (P_6)$ $=$ $1$, and $I_1 \circ p (P_7)$ $=$ $1$.  The diagram $P_7$ appears in \cite{HY}.  
\end{example}
\begin{figure}[h]
\begin{picture}(0,0)
\put(28,-10){$P_5$}
\put(134,-10){$P_6$}
\put(270,-10){$P_7$}
\end{picture}
\includegraphics[width=12cm]{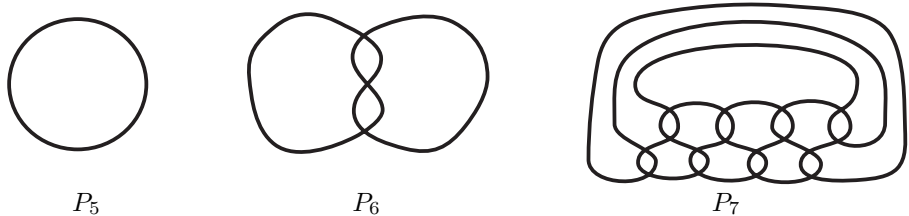}
\caption{Examples of knot projections.}\label{knotProjections}
\end{figure}

\section{Triviality of knot projections under weak (1, 3) homotopy}
Let us call the knot projection of $P_5$ in Fig.~\ref{knotProjections} {\it{the trivial diagram}}.  If a knot projection $P$ and the trivial diagram can be related by a finite sequences of $H_1$ and $H_3$ moves, the knot projection $P$ is described as {\it{trivial under weak (1, 3) homotopy}}.  If a diagram the same as $P_5$ is a knot diagram, we also call the knot diagram {\it{the trivial diagram}}.  
\begin{corollary}\label{trivialCondition}
Let $P$ be an arbitrary knot projection.  The knot projection $P$ is trivial under weak (1, 3) homotopy if and only if $P$ and the trivial diagram can be related by a finite sequence consisting of the first homotopy moves.  
\end{corollary}
\begin{proof}
By Theorem \ref{theorem1}, there exists the map from weak (1,3) homotopy classes of knot projections to knot isotopy classes, which is denoted by $p$.  Then if $T$ is a weak (1, 3) homotopy classes of a knot projection containing the trivial diagram, then $p(T)$ is the unknot that is defined as the knot isotopy classes containing the trivial diagram. 

We will prove that the converse of the claim.  If $P$ is an arbitrary weak homotopy classes of an arbitrary knot projection, then $p(P)$ is a positive knot by the definition of $p$.  As is well known, if a positive knot diagram belongs to the isotopy class of the unknot, there exists a finite sequence of $\Omega_1$ moves (Fig.~\ref{reidemeister}) between the positive knot diagram and the trivial diagram (e.g. \cite{PT}).  Neglecting the information on overpasses and underpasses at double points, we can regard the sequence as one that consists of $H_1$ moves between a knot projection and the trivial diagram.  This completes the proof.  
\end{proof}
Theorem \ref{theorem1} and Corollary \ref{trivialCondition} imply the following.  
\begin{corollary}
Let $p$ the map defined by Theorem \ref{theorem1} from a weak (1,3) homotopy class of knot projections to a knot isotopy class.  If the weak (1, 3) homotopy class containing the trivial diagram is denoted by $T$ and the unknot is denoted by $U$, then
\[p(T) = U~{\text{and}}~p^{-1}(U) = T.  \] 
\end{corollary}
\begin{proof}
The first equation is assured by Theorem \ref{theorem1} and the second equation is obtained from Corollary \ref{trivialCondition}.  
\end{proof}

\section*{Acknowledgments}
The authors would like to thank Professor Kouki Taniyama for his fruitful comments.  The work on this paper by N. Ito was partially supported by JSPS KAKENHI Grant Number 23740062.

\vspace{3mm}

Waseda Institute for Advanced Study, 1-6-1, Nishi Waseda Shinjuku-ku Tokyo 169-8050, Japan\\
{\it{E-mail address}}: {\texttt{noboru@moegi.waseda.jp}}

\vspace{3mm}
Department of Mathematics, School of Education, Waseda University, 1-6-1 Nishi Waseda, Shinjuku-ku, Tokyo, 169-8050, Japan\\
{\it{E-mail address}}: {\texttt{max-drive@moegi.waseda.jp}}

\end{document}